\documentclass{elsarticle}

\usepackage[english]{babel}
\usepackage{dsfont} 
\usepackage{graphicx}
\usepackage{color}
\usepackage[hyperindex,breaklinks]{hyperref}
\usepackage{enumerate}
\usepackage{amssymb,empheq} 
\usepackage{amsthm} 
\usepackage{cleveref}

\allowdisplaybreaks
\newtheorem{thm}{Theorem}[section]

\newtheorem{lem}[thm]{Lemma}
\newtheorem{prop}[thm]{Proposition}
\newtheorem{cor}[thm]{Corollary}

\theoremstyle{definition}

\newtheorem{dfn}[thm]{Definition}
\newtheorem{exm}[thm]{Example}

\theoremstyle{remark}

\newcommand{\exmsymbol}{\hfill$\circ$}

\newcommand{\cset}{\mathds{C}}

\newcommand{\nset}{\mathds{N}}

\newcommand{\rset}{\mathds{R}}

\newcommand{\diff}{\mathrm{d}}

\newcommand{\pos}{\mathrm{Pos}}

\newcommand{\supp}{\mathrm{supp}\,}

\newcommand{\id}{\mathrm{id}}
\newcommand{\one}{\mathds{1}}

\newcommand{\fD}{\mathfrak{D}}
\newcommand{\fd}{\mathfrak{d}}

\newcommand{\fg}{\mathfrak{g}}

\author{Philipp J.\ di Dio}
\address{Department of Mathematics and Statistics, University of Konstanz, Universit\"atsstra{\ss}e 10, D-78464 Konstanz, Germany}
\address{Department of Computer and Information Science, University of Konstanz, Universit\"atsstra{\ss}e 10, D-78464 Konstanz, Germany}
\address{Zukunftskolleg, Universtity of Konstanz, Universit\"atsstra{\ss}e 10, D-78464 Konstanz, Germany}
\address{philipp.didio@uni-konstanz.de}

\journal{working project}

\title{Linear Operators $T:\rset[x_1,\dots,x_n]\to\rset[x_1,\dots,x_n]$\\ and $K$-Positivity Preserver: A Short Review}

\begin{document}

\begin{abstract}
In the current short review we present the latest developments on linear maps $T:\rset[x_1,\dots,x_n]\to\rset[x_1,\dots,x_n]$, especially of $K$-positivity preserver, i.e., $Tp\geq 0$ on $K\subseteq\rset^n$ for all $p\in\rset[x_1,\dots,x_n]$ with $p\geq 0$ on $K$.
\end{abstract}


\maketitle


\section{Introduction}

The vector space $\cset[x_1,\dots,x_n]$ and the cone
\[\pos(K) := \big\{f\in\rset[x_1,\dots,x_n] \,\big|\, f\geq 0\ \text{on}\ K\big\}\]
of non-negative polynomials on a closed set $K\subseteq\rset^n$ belong to the most studied objects in mathematics and it is the main topic of (real) algebraic geometry.
However, much less studied are linear maps
\[T:\rset[x_1,\dots,x_n]\to\rset[x_1,\dots,x_n]\]
especially with
\[T\pos(K)\subseteq\pos(K),\]
see e.g.\ \cite{guterman08,netzer10,borcea11,didio24posPresConst,
didio25KPosPresGen,didio25hadamardLanger,didio25PosToSoS}.
In the current short review we give an overview over the very recent developments in \cite{didio24posPresConst,didio25KPosPresGen,didio25hadamardLanger,didio25PosToSoS}.

\section{Preliminaries}
\label{sec:prelim}

\subsection{Moment Functionals and Sequences}

Let $n\in\nset$, let $K\subseteq\rset^n$, and let $s=(s_\alpha)_{\alpha\in\nset_0^n}$ be a sequence.
We call $s$ a \emph{$K$-moment sequence}, if there exists a measure $\mu$ with $\supp\mu\subseteq K$ and
\[s_\alpha = \int_K x^\alpha~\diff\mu(x)\]
for all $\alpha\in\nset_0^n$.
Let
\[L:\rset[x_1,\dots,x_n]\to\rset\]
be a linear functional.
We call $L$ a $K$-moment functional, if there exists a measure $\mu$ with $\supp\mu\subseteq K$ and
\[L(p) = \int_K p(x)~\diff\mu(x)\]
for all $p\in\rset[x_1,\dots,x_n]$.

\subsection{Frech\'et and LF-Spaces}

A \textit{Fr\'echet space} is a metrizable, complete, and locally convex topological vector space.
For our purposes it is sufficient to have the following example.

\begin{exm}[see e.g.\ {\cite[pp.\ 91--92, Ex.\ III]{treves67}}]
The vector space
\[\cset[[x_1,\dots,x_n]]\]
of formal power series is a Fr\'echet space with the family
\[\{p_d\}_{d\in\nset_0}\]
of semi-norms
\[p_{d}(f) = \sup_{\alpha\in\nset_0^n:|\alpha|\leq d} |c_\alpha|\]
of $f = \sum_{\alpha\in\nset_0^n} c_\alpha\cdot x^\alpha$.
\exmsymbol
\end{exm}

Clearly,
\[\cset[[x_1,\dots,x_n]] \cong \cset^{\nset_0^n}\]
and hence the space $\cset^{\nset_0^n}$ of complex sequences is a Frech\'et space.

The vector space $\cset[x_1,\dots,x_n]$ is a LF-space \cite[p.\ 130, Ex.\ I]{treves67}.
The general definition of a LF-space can be found e.g.\ in \cite[p.\ 126]{treves67} or \cite[Ch.\ II, 6.4]{schaef99}.
For our purposes it is sufficient to understand that a sequence of polynomials $(p_i)_{i\in\nset_0}$ converges to some $p\in\cset[x_1,\dots,x_n]$, i.e., $p_i\to p$ as $i\to\infty$, if and only if $\sup_{i\in\nset_0} \deg p_i < \infty$ and each $x^\alpha$-coefficient of $p_i$ converges to the $x^\alpha$-coefficient of $p$.

Consequently, the spaces $\cset[[x_1,\dots,x_n]]$ and $\cset[x_1,\dots,x_n]$ are dual to each other via
\[\langle f,p\rangle := \sum_{\alpha\in\nset_0^n} c_\alpha\cdot p_\alpha\]
for
\[f = \sum_{\alpha\in\nset_0^n} c_\alpha\cdot x^\alpha \in \cset[[x_1,\dots,x_n]]
\quad\text{and}\quad
p = \sum_{\alpha\in\nset_0^n} p_\alpha\cdot x^\alpha \in \cset[x_1,\dots,x_n].\]

The interested reader can find more about Fr\'echet and LF-spaces in e.g.\ \cite{treves67,schaef99}.

\subsection{Regular Frech\'et Lie Groups}

\begin{dfn}[see e.g.\ {\cite[p.\ 63, Dfn.\ 1.1]{omori97}}]\label{dfn:frechetLieGroup}
We call $(G,\,\cdot\,)$ a \textit{regular Fr\'echet Lie group} if the following conditions are fulfilled:
\begin{enumerate}[(i)]
\item $G$ is an infinite dimensional smooth Fr\'echet manifold.

\item $(G,\,\cdot\,)$ is a group.

\item The map $G\times G\to G$, $(A,B)\mapsto A\cdot B^{-1}$ is smooth.

\item The \textit{Fr\'echet Lie algebra} $\fg$ of $G$ is isomorphic to the tangent space $T_e G$ of $G$ at the unit element $e\in G$.

\item $\exp:\fg\to G$ is a smooth mapping such that
$\left.\frac{\diff}{\diff t} \exp(t u)\right|_{t=0} = u$
for all $u\in\fg$.

\item The space $C^1(G,\fg)$ of $C^1$-curves in $G$ coincides with the set of all $C^1$-curves in $G$ under the Fr\'echet topology.
\end{enumerate}
\end{dfn}

The interested reader can find more on infinite dimensional manifolds, differential calculus, Lie groups, and Lie algebras in e.g.\ \cite{omori97,schmed23}.

\section{Representations of Linear Operators on $\rset[x_1,\dots,x_n]$}
\label{sec:repr}

To study linear operators on $\rset[x_1,\dots,x_n]$, we have at first look at how they look like, i.e., we deal in this section with the different kinds of representations.
At first we have the \emph{canonical representation}.

\begin{thm}[folklore, \emph{canonical representation}]\label{thm:canRepr}
Let $n\in\nset$ and let
\[T:\rset[x_1,\dots,x_n]\to\rset[x_1,\dots,x_n]\]
be linear.
Then, for all $\alpha\in\nset_0^n$, there exist unique polynomials $q_\alpha\in\rset[x_1,\dots,x_n]$ such that
\begin{equation}\label{eq:canRepr}
T = \sum_{\alpha\in\nset_0^n} q_\alpha\cdot\partial^\alpha.
\end{equation}
\end{thm}

A proof of the proceeding result can be found e.g.\ in \cite{netzer10}.
With the representation (\ref{eq:canRepr}), especially degree preserving maps, i.e.,
\[T\rset[x_1,\dots,x_n]_{\leq d} \subseteq \rset[x_1,\dots,x_n]_{\leq d}\]
for all $d\in\nset_0$, can easily be characterized, as we see in the next result.

\begin{lem}[{\cite[Lem.\ 5.1]{didio25KPosPresGen}}]
Let $n\in\nset$ and let
\[T:\rset[x_1,\dots,x_n]\to\rset[x_1,\dots,x_n] \quad\text{with}\quad T = \sum_{\alpha\in\nset_0^n} q_\alpha\cdot\partial^\alpha\]
be linear with unique $q_\alpha\in\rset[x_1,\dots,x_n]$ for all $\alpha\in\nset_0^n$.
Then the following are equivalent:
\begin{enumerate}[\;\, (i)]
\item $T\rset[x_1,\dots,x_n]_{\leq d} \subseteq \rset[x_1,\dots,x_n]_{\leq d}$ for all $d\in\nset_0$.

\item $\deg q_\alpha\leq |\alpha|$ for all $\alpha\in\nset_0^n$.
\end{enumerate}
\end{lem}

A second representation can be view as a consequence or a version of the Schwartz Kernel Theorem, see e.g.\ \cite[Ch.\ 50]{treves67}.

\begin{lem}[{\cite[Rem.\ 2.3]{didio25hadamardLanger}}]
Let $n\in\nset$ and let
\[T:\rset[x_1,\dots,x_n]\to\rset[x_1,\dots,x_n]\]
be linear.
Then, for all $i\in\nset_0$, there are linear functionals
\[l_i:\rset[x_1,\dots,x_n]\to\rset\]
and polynomials $p_i\in\rset[x_1,\dots,x_n]$ such that
\begin{equation}\label{eq:tensorRepr}
T = \sum_{i\in\nset_0} l_i\cdot p_i, \quad\text{i.e.,}\quad Tf = \sum_{i\in\nset_0} l_i(f)\cdot p_i
\end{equation}
for all $f\in\rset[x_1,\dots,x_n]$.
\end{lem}

Among the linear maps $T$ on $\rset[x_1,\dots,x_n]$, the linear operators with
\begin{equation}\label{eq:diagonalMaps}
Tx^\alpha = t_\alpha\cdot x^\alpha \quad\text{and}\quad t_\alpha\in\rset
\end{equation}
for all $\alpha\in\nset_0^n$ are called \emph{diagonal}.
Clearly, a diagonal operator is uniquely determined by its diagonal sequence $t=(t_\alpha)_{\alpha\in\nset_0^n}$ and every real sequence $t = (t_\alpha)_{\alpha\in\nset_0^n}$ gives a well-defined diagonal operator.
They have a special structure and therefore special properties, general linear maps on $\rset[x_1,\dots,x_n]$ do not possess.
Inserting (\ref{eq:diagonalMaps}) into \Cref{thm:canRepr} gives the following.

\begin{cor}[{\cite[Rem.\ 4.2]{didio25hadamardLanger}}]\label{thm:diagonalRepr}
Let $n\in\nset$ and let
\[T:\rset[x_1,\dots,x_n]\to\rset[x_1,\dots,x_n]\]
be a linear operator.
Then the following are equivalent:
\begin{enumerate}[(i)]
\item $T$ is a diagonal operator, i.e.,
\[Tx^\alpha = t_\alpha\cdot x^\alpha \quad\text{and}\quad t_\alpha\in\rset\]
for all $\alpha\in\nset_0^n$.

\item The operator $T$ is of the form
\begin{equation}\label{eq:diagRepr1}
T = \sum_{\alpha\in\nset_0^n} \frac{c_\alpha}{\alpha!}\cdot x^\alpha\cdot \partial^\alpha
\end{equation}
with a unique real sequence $c = (c_\alpha)_{\alpha\in\nset_0^n}$.
\end{enumerate}
If one of the equivalent statements (i) or (ii) holds, then the diagonal sequence $t=(t_\alpha)_{\alpha\in\nset_0^n}$ and the sequence $c=(c_\alpha)_{\alpha\in\nset_0^n}$ of coefficients $c_\alpha$ fulfill the relations
\begin{equation*}\label{eq:talphacalphaRelations}
t_\alpha = \sum_{\beta\in\nset_0^n:\ \beta\preceq\alpha} \binom{\alpha}{\beta}\cdot c_\beta \qquad\text{and}\qquad c_\alpha = \sum_{\beta\in\nset_0^n:\ \beta\preceq\alpha} (-1)^{|\alpha-\beta|}\cdot \binom{\alpha}{\beta}\cdot t_\beta
\end{equation*}
for all $\alpha\in\nset_0^n$.
\end{cor}

\section{$K$-Positivity Preservers}
\label{sec:KposPres}

We now come to our main topic: \emph{$K$-positivity preserver}.

\begin{dfn}
Let $n\in\nset$, let $K\subseteq\rset^n$ be closed, and let
\[T:\rset[x_1,\dots,x_n]\to\rset[x_1,\dots,x_n]\]
be linear.
If
\[T\pos(K)\subseteq\pos(K),\]
then we call $T$ a \emph{$K$-positivity preserver}.
If $K=\rset^n$, then, for short, we call $T$ a \emph{positivity preserver}.
\end{dfn}

If $K$ is empty ($K=\emptyset$), then
\[\pos(\emptyset) = \rset[x_1,\dots,x_n]\]
and hence every linear map
\[T:\rset[x_1,\dots,x_n]\to\rset[x_1,\dots,x_n]\]
is a $\emptyset$-positivity preserver.
$K=\emptyset$ is trivial and we can therefore assume $K\neq\emptyset$.

\begin{exm}\label{exm:KposPres}
Let $n,k\in\nset$, let $K\subseteq\rset^n$ be closed and non-empty, let $p_1,\dots,p_k\in\pos(K)$, and let 
\[l_1,\dots,l_k:\rset[x_1,\dots,x_n]\to\rset\]
be $K$-moment functionals.
Then
\[T:\rset[x_1,\dots,x_n]\to\rset[x_1,\dots,x_n],\quad f\mapsto Tf := \sum_{i=1}^k l_i(f)\cdot p_i\]
is a $K$-positivity preserver.
\exmsymbol
\end{exm}

Not every $K$-positivity preserver is of the form given in \Cref{exm:KposPres}.

\begin{exm}[{\cite[Rem.\ 2.3]{didio25hadamardLanger}}]\label{exm:nonFiniteSumRepr}
Let $n\in\nset$ and
\[\id:\rset[x_1,\dots,x_n]\to\rset[x_1,\dots,x_n], \quad f\mapsto f\]
be the identity on $\rset[x_1,\dots,x_n]$.
Clearly, $\id$ is a $K$-positivity preserver for any (non-empty) $K\subseteq\rset^n$.
Let $K=\rset^n$.
Then $\id$ does not possess a representation as in \Cref{exm:KposPres}.
To see this, assume that $\id$ has such a representation, i.e., there exist a $k\in\nset$, polynomials $p_1,\dots,p_k\in\pos(K)$, and $\rset^n$-moment functionals
\[l_1,\dots,l_k:\rset[x_1,\dots,x_n]\to\rset\]
such that
\[\id = \sum_{i=1}^k l_i\cdot p_i.\]
But since $\id(x_1+a)^2 = (x_1+a)^2$, $(x_1+a)^2$ is an extreme point of $\pos(\rset^n)$, and since $l_i((x_i+a)^2)\geq 0$ for any $a\in\rset$, we have that $(x_1+a)^2$ must be contained in $\{p_1,\dots,p_k\}$.
But since $\{p_1,\dots,p_k\}$ is finite (countable) and $\rset$ is uncountable, we arrived at a contradiction, i.e., $\id$ is not of a form as in \Cref{exm:KposPres}.
\exmsymbol
\end{exm}

A much simpler argument as in \Cref{exm:nonFiniteSumRepr} and therefore also simpler as in \cite{didio25hadamardLanger} is, that any $K$-positivity preserver as in \Cref{exm:KposPres} has only finite dimensional image.
Hence, any $K$-positivity preserver with infinite dimensional image can not be represented as in \Cref{exm:KposPres}, since all $l_i\geq 0$ on $\pos(K)$.

Before we give a characterization of $K$-positivity preservers, we need the following definition.

\begin{dfn}\label{dfn:Ty}
Let $q_\alpha\in\rset[x_1,\dots,x_n]$ for all $\alpha\in\nset_0^n$ and let
\[T = \sum_{\alpha\in\nset_0^n} q_\alpha\cdot\partial^\alpha\]
be a linear operator.
For $y\in\rset^n$, we define the linear map
\begin{equation}\label{eq:Ty}
T_y:\rset[x_1,\dots,x_n]\to\rset[x_1,\dots,x_n]
\quad\text{by}\quad
T_y := \sum_{\alpha\in\nset_0^n} q_\alpha(y)\cdot\partial^\alpha.
\end{equation}
\end{dfn}

We can now give the following characterization of $K$-positivity preserver.

\begin{thm}[{\cite[Main Thm.\ 3.5]{didio25KPosPresGen}}]\label{thm:kPosPresCara}
Let $n\in\nset$, $K\subseteq\rset^n$ be closed, and
\[T:\rset[x_1,\dots,x_n]\to\rset[x_1,\dots,x_n] \quad\text{with}\quad T = \sum_{\alpha\in\nset_0^n} q_\alpha\cdot\partial^\alpha\]
be linear, $q_\alpha\in\rset[x_1,\dots,x_n]$ for all $\alpha\in\nset_0^n$.
Then the following are equivalent:
\begin{enumerate}[(i)]
\item $T$ is a $K$-positivity preserver.

\item For all $y\in K$, the sequence $(\alpha!\cdot q_\alpha(y))_{\alpha\in\nset_0^n}$ is a $(K-y)$-moment sequence.
\end{enumerate}
If one of the equivalent conditions (i) or (ii) holds, then, for any $y\in K$,
\[(T_y f)(y) = \int f(x+y)~\diff\mu_y(x)\]
for all $f\in\rset[x_1,\dots,x_n]$, where $\mu_y$ is a representing measure of the $(K-y)$-moment sequence $(\alpha!\cdot q_\alpha(y))_{\alpha\in\nset_0^n}$.
\end{thm}

As mentioned before, if $K=\emptyset$, then any linear map on $\rset[x_1,\dots,x_n]$ is a $K$-positivity preserver, i.e., (i) in \Cref{thm:kPosPresCara} is always fulfilled, and (ii) in \Cref{thm:kPosPresCara} is the empty condition, i.e., it is also always fulfilled.

From \Cref{thm:kPosPresCara} with $K=\rset^n$ we immediately gain the classical characterization of $\rset^n$-positivity preserver of Julius Borcea.

\begin{cor}[{\cite[Thm.\ 3.1]{borcea11}}]\label{cor:borcea}
Let
\[T = \sum_{\alpha\in\nset_0^n} q_\alpha\cdot\partial^\alpha\]
be with $q_\alpha\in\rset[x_1,\dots,x_n]$.
Then the following are equivalent:
\begin{enumerate}[(i)]
\item $T$ is a $\rset^n$-positivity preserver.

\item For every $y\in\rset^n$, $(\alpha!\cdot q_\alpha(y))_{\alpha\in\nset_0^n}$ is a $\rset^n$-moment sequence.
\end{enumerate}
If one of the equivalent conditions holds, then for every $y\in\rset^n$ the map $T_y$ is given by
\[(T_y p)(x) := \int_{\rset^n} p(x+z)~\diff\mu_y(z)\]
for $p\in\rset[x_1,\dots,x_n]$, where $\mu_y$ is a representing measure of $(\alpha!\cdot q_\alpha(y))_{\alpha\in\nset_0^n}$.
\end{cor}

\section{Generators $A$ of Linear Maps $T$: $T=e^{A}$}
\label{sec:gen}

Among the linear maps
\[T:\rset[x_1,\dots,x_n]\to\rset[x_1,\dots,x_n]\]
a special class are the maps of the form
\[T = e^{A}\]
for a linear map
\[A:\rset[x_1,\dots,x_n]\to\rset[x_1,\dots,x_n].\]
Not every $T$ can be written as $T = e^A$ and not every $A$ gives a well-defined $e^A$.
Hence, at first we have to determine, which $A$ give well-defined $e^A$.

\begin{dfn}\label{dfn:fg}
Let $n\in\nset$.
We define
\begin{multline*}
\fg := \big\{A:\rset[x_1,\dots,x_n]\to\rset[x_1,\dots,x_n] \,\big|\, A\ \text{linear and}\\\ e^{tA}:\rset[x_1,\dots,x_n]\to\rset[x_1,\dots,x_n]\ 
\text{well-defined for all}\ t\in\rset\big\}.
\end{multline*}
\end{dfn}

While $\fg$ is an abstract set, the following gives a characterization.

\begin{thm}[{\cite[Thm.\ 4.4]{didio25PosToSoS}}]\label{thm:fgCaracterization}
Let $n\in\nset$ and
\[A:\rset[x_1,\dots,x_n]\to\rset[x_1,\dots,x_n]\]
be linear.
The following are equivalent:
\begin{enumerate}[(i)]
\item $A\in\fg$.

\item $\displaystyle \sup_{k\in\nset_0} \deg A^k x^\alpha < \infty$ for all $\alpha\in\nset_0^n$.

\item $\displaystyle \sup_{k\in\nset_0} \deg A^k p < \infty$ for all $p\in\rset[x_1,\dots,x_n]$.

\item For all $i\in\nset_0$, there exist subvector spaces $V_i\subseteq\rset[x_1,\dots,x_n]$ with
\begin{enumerate}[(a)]
\item $\dim V_i < \infty$ for all $i\in\nset_0$,

\item $\displaystyle \bigcup_{i\in\nset_0} V_i = \rset[x_1,\dots,x_n]$, and

\item $AV_i \subseteq V_i$ for all $i\in\nset_0$.
\end{enumerate}
\end{enumerate}
\end{thm}

An immediate consequence of \Cref{thm:fgCaracterization} is the following.

\begin{prop}\label{prop:eigenvalueVector}
Let $n\in\nset$ and $A\in\fg$.
Then $A$ has a complex eigenvalue.
\end{prop}
\begin{proof}
By \Cref{thm:fgCaracterization}, there exists a $V\subseteq\rset[x_1,\dots,x_n]$ with
\[\dim V < \infty \quad\text{and}\quad AV\subseteq V,\]
i.e., $A|_V:V\to V$ is a linear operator on a finite dimensional vector space (matrix) and hence has at least one (complex) eigenvalue.
\end{proof}

Since by \ref{thm:fgCaracterization} it is sufficient to know every $A\in\fg$ only on finite dimensional invariant subspaces $V_i$, we have the following.

\begin{cor}[{\cite[Cor.\ 5.6]{didio25KPosPresGen}} and {\cite[Cor.\ 4.5]{didio25PosToSoS}}]
Let $A\in\fg$.
Then
\[\exp A = \sum_{k\in\nset_0^n} \frac{A^k}{k!} = \lim_{k\to\infty} \left( \one + \frac{A}{k} \right)^k = \lim_{k\to\infty} \left( \one - \frac{A}{k} \right)^{-k}.\]
\end{cor}

\section{Generators $A$ of $K$-Positivity Preserving Semi-Groups $(e^{tA})_{t\geq 0}$}
\label{sec:posGen}

While we have characterized all linear maps $A$ on $\rset[x_1,\dots,x_n]$ such that $e^A$ is well-defined, we now want to know, which of these $A$ are generators of a $K$-positivity preserving semi-group $(e^{tA})_{t\geq 0}$.

\begin{dfn}
Let $n\in\nset$, let $K\subseteq\rset^n$ be closed, and let $A\in\fg$.
We say $A$ is a \emph{generator of a $K$-positivity preserving semi-group} $(e^{tA})_{t\geq 0}$, if $e^{tA}$ is a $K$-positivity preserver for all $t\geq 0$.
\end{dfn}

We start with the simpler case of constant coefficients.

\begin{dfn}
Let $n\in\nset$.
We define
\[\fD_c := \left\{\sum_{\alpha\in\nset_0^n} q_\alpha\cdot \partial^\alpha \,\middle|\, q_\alpha\in\rset,\ q_0 = 1\right\}\]
and
\[\fd_c := \left\{ \sum_{\alpha\in\nset_0^n\setminus\{0\}} d_\alpha\cdot\partial^\alpha \,\middle|\, d_\alpha\in\rset\ \text{for all}\ \alpha\in\nset_0^n\setminus\{0\}\right\}.\]
\end{dfn}

\begin{thm}[{\cite[Thm.\ 2.17]{didio24posPresConst}}]\label{thm:mainFrechetLieGroups}
Let $n\in\nset$.
Then $(\fD_c,\,\cdot\,)$ as a Fr\'echet space is a commutative regular Fr\'echet Lie group with the commutative Fr\'echet Lie algebra $(\fd_c,\,\cdot\,,+)$.
The exponential map
\[\exp:\fd_c\to\fD_c,\quad A\mapsto \sum_{k\in\nset_0} \frac{A^k}{k!}\]
is smooth and bijective with the smooth and bijective inverse
\[\log:\fD_c\to\fd_c,\quad A\mapsto -\sum_{k\in\nset} \frac{(\one-A)^k}{k}.\]
\end{thm}

For the constant coefficient case, we can characterize all generators of $\rset^n$-positivity preserving semi-groups.

\begin{dfn}\label{dfn:fd+}
Let $n\in\nset$.
We define the set
\[\fD_{c,+} := \{A\in\fD_c \,|\, A\ \text{is a positivity preserver}\}\]
of all \emph{positivity preservers with constant coefficients} and we define the set
\[\fd_{c,+}:=\left\{A\in\fd_c\,\middle|\,\exp(tA)\in\fD_{c,+}\ \text{for all}\ t\geq 0\right\}\]
of all \emph{generators of positivity preservers with constant coefficients}.
\end{dfn}

For the characterization, we need the following result.

\begin{thm}[{\cite[Main Thm.\ 4.7]{didio24posPresConst}}]\label{thm:infinitelyDivisibleMeasure}
Let $n\in\nset$.
Then the following are equivalent:
\begin{enumerate}[(i)]
\item $A\in\fd_{c,+}$.

\item $\exp A$ has an infinitely divisible representing measure.

\item $\exp(tA)$ has an infinitely divisible representing measure for some $t>0$.

\item $\exp(tA)$ has an infinitely divisible representing measure for all $t>0$.
\end{enumerate}
\end{thm}

With \Cref{thm:infinitelyDivisibleMeasure} and the characterization of infinitely divisible measures by the L\'evy--Khinchin formula we get the following.

\begin{thm}[{\cite[Main Thm.\ 4.11]{didio24posPresConst}}]\label{thm:mainPosGenerators}
Let $n\in\nset$.
Then the following are equivalent:
\begin{enumerate}[(i)]
\item $\displaystyle A = \sum_{\alpha\in\nset_0^n\setminus\{0\}} \frac{a_\alpha}{\alpha!}\cdot\partial^\alpha \in\fd_{c,+}$.

\item There exists a symmetric matrix $\Sigma = (\sigma_{i,j})_{i,j=1}^n\in\rset^n$ with $\Sigma\succeq 0$, a vector $b = (b_1,\dots,b_n)^T\in\rset^n$, and a measure $\nu$ on $\rset^n$ with
\[\nu(\{0\})=0 \qquad\text{and}\qquad \int_{\rset^n} |x^\alpha|~\diff\nu(x)<\infty\]
for $\alpha\in\nset_0^n$ with $|\alpha|\geq 2$ such that
\begin{align*}
a_{e_i} &= b_i + \int_{\|x\|_2\geq 1} x_i~\diff\nu(x) && \text{for all}\ i=1,\dots,n,\\
a_{e_i+e_j} &= \sigma_{i,j} + \int_{\rset^n} x^{e_i + e_j}~\diff\nu(x) && \text{for all}\ i,j=1,\dots,n,
\intertext{and}
a_\alpha &= \int_{\rset^n} x^\alpha~\diff\nu(x) &&\text{for all}\ \alpha\in\nset_0^n\ \text{with}\ |\alpha|\geq 3.
\end{align*}
\end{enumerate}
\end{thm}

The measure $\nu$ in \Cref{thm:mainPosGenerators} is the L\'evy measure from the characterization of the infinitely divisible measure by the L\'evy--Khinchin formula.

For the case of non-constant coefficients we have the following result.
It follows from the constant coefficient case by non-trivial arguments.

\begin{thm}[{\cite[Main Thm.\ 6.12]{didio25KPosPresGen}}]\label{thm:PosGeneratorsGeneral}
Let
\[A = \sum_{\alpha\in\nset_0^n} \frac{a_\alpha}{\alpha!}\cdot\partial^\alpha\]
be with $a_\alpha\in\rset[x_1,\dots,x_n]_{\leq |\alpha|}$ for all $\alpha\in\nset_0^n$.
Then the following are equivalent:
\begin{enumerate}[(i)]
\item $e^{tA}$ is a $\rset^n$-positivity preserver for all $t\geq 0$.

\item For every $y\in\rset^n$, there exist a symmetric matrix $\Sigma(y) = (\sigma_{i,j}(y))_{i,j=1}^n$ with real entries such that $\Sigma(y)\succeq 0$, a vector $b(y) = (b_1(y),\dots,b_n(y))^T\in\rset^n$, a constant $a_0\in\rset$, and a $\sigma$-finite measure $\nu_y$ on $\rset^n$ with
\[\nu_y(\{0\})=0 \qquad\text{and}\qquad \int_{\rset^n} |x^\alpha|~\diff\nu_y(x)<\infty\]
for all $\alpha\in\nset_0^n$ with $|\alpha|\geq 2$ such that
\begin{align*}
a_{e_i}(y) &= b_i(y) + \int_{\|x\|_2\geq 1} x_i~\diff\nu_y(x) && \text{for}\ i=1,\dots,n,\\
a_{e_i+e_j}(y) &= \sigma_{i,j}(y) + \int_{\rset^n} x^{e_i + e_j}~\diff\nu_y(x) && \text{for all}\ i,j=1,\dots,n,
\intertext{and}
a_\alpha(y) &= \int_{\rset^n} x^\alpha~\diff\nu_y(x) &&\text{for}\ \alpha\in\nset_0^n\ \text{with}\ |\alpha|\geq 3.
\end{align*}
\end{enumerate}
\end{thm}

While \Cref{thm:PosGeneratorsGeneral} solves the case of generators of $\rset^n$-positivity preserving semi-groups which are degree preserving, the general case of an explicit description is still open.

\section{Summary}
\label{sec:sum}

While
\[\cset[x_1,\dots,x_n] \quad\text{and}\quad \pos(K)\]
are very well-studied in (real) algebraic geometry, surprisingly, even the linear operators
\[T:\rset[x_1,\dots,x_n]\to\rset[x_1,\dots,x_n]\]
are very little studied.
Especially under the condition that $T$ is a $K$-positivity preserver, very little was known so far.
That these operators are so little studied is surprising, since in every linear algebra course at first vector spaces $V$ are introduced and investigated (dimension, basis, subspaces etc.) and then the actual linear algebra starts when one looks at linear operators
\[T:V\to V\]
for finite-dimensional $V$, i.e., $T$ are matrices.
This step going from the vector spaces $\cset[x_1,\dots,x_n]$ and cones $\pos(K)$ to linear maps on them is so far hardly considered.
There are possibly two reasons for that.
Firstly, (real) algebraic geometry heavily relies on multiplication (in $\cset[x_1,\dots,x_n]$) and in general we have
\[T(f\cdot g) \neq Tf \cdot Tg,\]
i.e., we lose the very basic operation that makes algebra into algebra and not just the study of vector spaces.
Secondly, $\cset[x_1,\dots,x_n]$ is infinite dimensional and hence topological aspects come into play, which are usually not needed by algebraists.

These limitations are only overcome by looking from an (unusual) functional analytic point of view on these matters.
While moments and measures are deeply connected to $\pos(K)$ via duality, ``ordinary'' functional analysis works on Hilbert spaces and finite dimensional Lie groups.
The restriction to work only on $\cset[x_1,\dots,x_n]$ and not being allowed to make it into a Hilbert space (like is usually done by working with moments) forces us to drop the rich theory of Hilbert spaces and go to Fr\'echet spaces, LF-spaces, and the regular Fr\'echet Lie groups.
Here, more technical details have to be considered since we no longer have only one norm as in a Hilbert space but we have to work with families of semi-norm.

However, the current review about the works \cite{didio24posPresConst,didio25KPosPresGen,didio25hadamardLanger,didio25PosToSoS} shows that still results can be obtained and that we possibly (and hopefully) stand here only at the very beginning of the investigation of linear operators
\[T:\rset[x_1,\dots,x_n]\to\rset[x_1,\dots,x_n].\]

\section*{Funding}

The author and this project are supported by the Deutsche Forschungs\-gemein\-schaft DFG with the grant DI-2780/2-1 and his research fellowship at the Zukunfts\-kolleg of the University of Konstanz, funded as part of the Excellence Strategy of the German Federal and State Government.


\providecommand{\bysame}{\leavevmode\hbox to3em{\hrulefill}\thinspace}
\providecommand{\MR}{\relax\ifhmode\unskip\space\fi MR }
\providecommand{\MRhref}[2]{%
  \href{http://www.ams.org/mathscinet-getitem?mr=#1}{#2}
}
\providecommand{\href}[2]{#2}

\end{document}